\documentclass[11pt]{article}

\advance\topmargin by -\headsep \hbadness5000 \emergencystretch=1cm
\everydisplay{\mathsurround 0pt} 

\RequirePackage[OT1]{fontenc}
\RequirePackage{amsthm,amsmath}
\RequirePackage[numbers]{natbib}
\RequirePackage[colorlinks,citecolor=blue,urlcolor=blue]{hyperref}
\usepackage{setspace}

\RequirePackage{array}
\RequirePackage{amssymb}
\RequirePackage{amsmath}
\RequirePackage{amsthm}
\RequirePackage{amsmath}
\RequirePackage{enumerate}
\RequirePackage{subfigure}
\RequirePackage{graphicx}

\newcommand{\R}{\mathbb R}

\newcommand{\Xb}{\mathbf{X}}
\newcommand{\xb}{\ensuremath{\mathbf{x}}}
\newcommand{\Yb}{\ensuremath{\mathbf{Y}}}

\newcommand{\Sb}{\mathbf{S}}
\newcommand{\thetab}{{\pmb \theta}}
\newcommand{\Sigb}{{\pmb \Sigma}}

\newcommand{\n}{^{(n)}}
\newcommand{\pr}{^{\prime}}
\newcommand{\ny}{n\rightarrow\infty}

\newtheorem{lemma}{Lemma}[section]
\newtheorem{proposition}{Proposition}[section]
\newtheorem{theorem}{Theorem}[section]

\begin{document}

\title{High-dimensional tests for spherical location \\ and spiked covariance
}

\author{Christophe {\sc Ley}\footnote{E-mail address: chrisley@ulb.ac.be; URL: http://homepages.ulb.ac.be/\~{}chrisley} \,, Davy {\sc Paindaveine}\footnote{E-mail address: dpaindav@ulb.ac.be; URL: http://homepages.ulb.ac.be/\~{}dpaindav} \,  and Thomas {\sc Verdebout}\footnote{E-mail address: thomas.verdebout@univ-lille3.fr; URL: http://perso.univ-lille3.fr/\~{}tverdebout} \vspace{0.5cm}\\
Universit\'e Libre de Bruxelles and Universit\' e 
Lille Nord de France}
\date{}

\maketitle

%

\begin{abstract}
Rotationally symmetric distributions on the $p$-dimensional unit hypersphere,  extremely popular in directional statistics, involve a location parameter~$\thetab$ that indicates the direction of the symmetry axis. The most classical way of addressing the spherical location problem~$\mathcal{H}_0:\thetab=\thetab_0$, with~$\thetab_0$  a fixed location,  is the so-called Watson test, which is based on the  sample mean of the observations. This test enjoys many desirable properties, but its implementation requires the sample size~$n$ to be large compared to the dimension~$p$. This is a severe limitation, since more and more problems nowadays involve high-dimensional directional data (e.g., in genetics or text mining). In this work, we therefore introduce a modified Watson statistic that can cope with high-dimensionality. We derive its asymptotic null distribution as both~$n$ and~$p$ go to infinity. This is achieved in a universal asymptotic framework that allows $p$ to go to infinity arbitrarily fast (or slowly) as a function of~$n$. We further show that our results also provide high-dimensional tests for a problem that has recently attracted much attention, namely that of testing that the covariance matrix of a multinormal distribution has a ``$\thetab_0$-spiked" structure. Finally, a Monte Carlo simulation study corroborates our asymptotic results.  
\end{abstract}

\section{Introduction}

\setcounter{equation} {0}

The technological advances and the ensuing new devices to collect and store data lead nowadays in many disciplines to data sets with very high dimension~$p$, often larger than the sample size~$n$. Consequently, there is a need for inferential methods that can deal with such high-dimensional data, and this has entailed  a huge activity related to high-dimen\-sional problems in the last decade. One- and multi-sample location problems have been investigated in 
\cite{SriFuj2006}, 
\cite{Sch2007}, 
\cite{CheQin2010}, 
\cite{Srietal2013}, 
and \cite{SriKub2013}, among others.  
Since the seminal paper \cite{LedWol2002}, problems related to covariance or scatter matrices have also been  thoroughly studied by several authors; see, e.g., \cite{Cheetal2010}, \cite{LiChe2012}, \cite{Onaetal2013} and \cite{JiaYan2013}.

In this paper, we are interested in high-dimensional \emph{directional} data, that is, in data lying on the unit hypersphere~$\mathcal{S}^{p-1}=\{\xb\in\R^p:\|\xb\|=\sqrt{\xb'\xb}=1\}$, with~$p$ large. Such data occur when only the direction of the observations and not their magnitude matters, and are extremely common, e.g., in magnetic resonance (\citealp{Dry2005}), gene-expression (\citealp{banerjee2003generative}), and 
text mining (\citealp{Banetal2005}). Inference for high-dimensional directional data has already been considered in several papers. For instance, \cite{BanGho2002,BanGho2004} and \cite{Banetal2005} investigate clustering methods in this context. Most asymptotic results from the literature, however, have been obtained as~$p$ goes to infinity, with $n$~fixed. This is the case of almost all results in \cite{Sta1982}, \cite{Wat1983a}, \cite{Wat1988}, and \cite{Dry2005}. To the best of our knowledge, the only $(n,p)$-asymptotic results available can be found in \cite{Dry2005}, \cite{CaJia12}, \cite{Caietal2013}, and \cite{PaiVer2013b}. However, 
\cite{Dry2005} imposes the stringent condition that~$p/n^2 \to \infty$ when studying the asymptotic behavior of the classical  pseudo-FvML location estimator (FvML here refers to \emph{Fisher-von Mises-Langevin} distributions; see below). \cite{CaJia12} and \cite{Caietal2013} consider various $(n, p)$-asymptotic regimes in the context of testing for uniformity on the unit sphere, but the tests to be used depend on the regime considered which makes practical implementation problematic. Finally,  \cite{PaiVer2013b} propose tests that are robust to the $(n, p)$-asymptotic regime considered; their tests, however, are sign procedures, hence are not based on sufficient statistics --- unlike the much more classical pseudo-FvML procedures.


In the present paper, we intend to overcome these limitations in the context of the spherical location problem,  one of the most fundamental problems in directional statistics. The natural distributional framework for this problem is provided by \emph{rotationally symmetric distributions} (see Section~\ref{prelim}), that form a semiparametric model, indexed by a finite-dimensional (location) parameter~$\thetab\in\mathcal{S}^{p-1}$ and an infinite-dimensional parameter~$F$. The spherical location problem consists in testing the null hypothesis $\mathcal{H}_0:\thetab=\thetab_0$ against alternative locations, where~$\thetab_0$ is a given unit vector and~$F$ remains unspecified. The  classical test for this problem is the so-called Watson test,  based on the sample mean of the observations; see  \cite{Wat1983b}. This test enjoys many desirable properties, and in particular is a \emph{pseudo-FvML} procedure~: in other words, it achieves optimality under FvML distributions, yet remains valid (in the sense that it meets the asymptotic nominal level constraint) under extremely mild assumptions on~$F$. 

Unfortunately, the Watson test cannot be used in the high-dimensional case, since its implementation crucially relies on fixed-$p$ asymptotic results. In view of the growing number of  high-dimensional directional data to analyze, this is a severe limitation. The aim of this paper hence is to define a modified Watson test statistic that can cope with high-dimensionality. We achieve this in such a way that asymptotic validity under virtually any rotationally symmetric distribution is maintained. Even better~: in contrast with earlier asymptotic investigations of high-dimensional pseudo-FvML procedures, our asymptotic results are ``universal'' in the sense that they only require that $p$~goes to infinity as~$n$ does ($p$ may go arbitrarily fast (or slowly) to infinity as a function of~$n$). Moreover, as a highly interesting by-product, we show that our procedure can be used to test the null hypothesis that the covariance matrix of a high-dimensional multinormal distribution is ``$\thetab_0$-spiked", meaning that it is of the form~$\Sigb= \sigma^2({\bf I}_p+ \lambda \thetab_0 \thetab_0\pr)$ for some~$\sigma^2,\lambda>0$ and~$\thetab_0\in\R^k$; see, e.g., \cite{johnstone2001} or the quite recent \cite{Onaetal2013} where this covariance structure has  been used as an alternative to sphericity. 

The outline of the paper is as follows. In Section~\ref{prelim}, we define the class of rotationally symmetric distributions and introduce the Watson test for spherical location. In Section~\ref{coolsec}, we propose a modified Watson test statistic and derive its asymptotic null distribution in the high-dimensional setting. We also prove that, in some cases, it is asymptotically equivalent to a sign test statistic. In Section~\ref{spike}, we show that the modified Watson test as well permits to test for a spiked covariance structure in multinormal distributions. A Monte Carlo simulation study is conducted in Section~\ref{simus}, while an Appendix collects the proofs of some technical lemmas.

\section{Rotational symmetry and the Watson test}
\label{prelim}


The distribution of the random $p$-vector~$\Xb$, with values on the unit hypersphere~$\mathcal{S}^{p-1}$, is \emph{rotationally symmetric} about location~$\thetab(\in\mathcal{S}^{p-1})$ if $\mathbf{O}\Xb$ is equal in distribution to~$\Xb$ for any orthogonal $p\times p$ matrix~$\mathbf{O}$ satisfying $\mathbf{O}\thetab=\thetab$; see \cite{Saw1978}.
Rotationally symmetric distributions are characterized by the location parameter~$\thetab$ and an infinite-dimensional parameter, the cumulative distribution function~$F$ of~$\Xb'\thetab$, hence they are of a semiparametric nature. The rotationally symmetric distribution associated with~$\thetab$ and~$F$ will be denoted as~$\mathcal{R}(\thetab,F)$ in the sequel. The most celebrated members of this family  are the Fisher-von Mises-Langevin distributions, corresponding to $F_{p,\kappa}(t)=c_{p,\kappa}\int_{-1}^t (1-s^2)^{(p-3)/2}\exp(\kappa s)\,ds$ ($t\in[-1,1]$), where~$c_{p,\kappa}$ is a normalization constant and~$\kappa(>0)$ is a \emph{concentration} parameter (the larger the value of~$\kappa$, the more concentrated about~$\thetab$ the distribution is); see \cite{MarJup2000} for further details.  


Let~$\Xb_1,\ldots,\Xb_n$ be a sequence of \mbox{i.i.d.} random unit vectors from~$\mathcal{R}(\thetab,F)$ and consider the problem of testing the null hypothesis~$\mathcal{H}_0:\thetab=\thetab_0$ against the alternative $\mathcal{H}_1:\thetab\neq\thetab_0$, where~$\thetab_0\in\mathcal{S}^{p-1}$ is fixed and~$F$ remains unspecified. Letting~$\bar{\Xb}:=\frac{1}{n}\sum_{i=1}^n \Xb_i$, the  classical test for this problem rejects the null for large values of the Watson statistic 
\begin{equation}
\label{Watson}
W_n 
:=
\frac{n (p-1)\bar{\Xb}\pr ({\bf I}_k- \thetab_0 \thetab_0\pr) \bar{\Xb}}{1-\frac{1}{n} \sum_{i=1}^n(\Xb_i\pr\thetab_0)^2} 
.
\end{equation}
Under very mild assumptions on~$F$, the fixed-$p$ asymptotic null distribution of~$W_n$ is chi-square with~$p-1$ degrees of freedom. The resulting test,~$\phi^W_n$ say, therefore rejects the null, at asymptotic level~$\alpha$, whenever~$W_n>\Psi_{p-1}^{-1}(1-\alpha)$, where $\Psi_{p-1}$ stands for the cumulative distribution function of the chi-square distribution with~$p-1$ degrees of freedom; see \cite{Wat1983b}.

Beyond achieving asymptotic level~$\alpha$ under virtually any rotationally symmetric distribution, $\phi^W_n$ is optimal --- more precisely, locally and asymptotically maximin, in the Le Cam sense --- when the underlying distribution is FvML; for details, we refer to~\cite{PaiVer2013}, where the asymptotic properties of~$\phi^W_n$  under local alternatives are  derived. Although~$\phi^W_n$ is based on the sample mean of the observations, these excellent power properties are not obtained at the expense of robustness, since observations by construction are on the unit hypersphere.  

Consequently, $\phi^W_n$ is a nice solution to the testing problem considered on all counts but one~: implementation is based on fixed-$p$ asymptotics, so that $\phi^W_n$ cannot be used when~$p$ is of the same order as, or even larger than,~$n$. The goal of the present work is therefore to derive a modified  Watson test, $\tilde{\phi}^W_n$ say, that can cope with high-dimensionality.

\section{A high-dimensional Watson test} 
\label{coolsec}

%
%

Consider the high-dimensional version of the testing problem $\mathcal{H}_0:\thetab=\thetab_0$ against $\mathcal{H}_1:\thetab\neq\thetab_0$, based on a triangular array of observations~$\Xb_{ni}$, $i=1,\ldots,n$, $n=1,2,\ldots,$ where $\Xb_{ni}$ takes values in~$\mathcal{S}^{p_n-1}$ and~$p_n$ goes to infinity with~$n$. In this section, we modify the Watson test statistic~$W_n$ in~(\ref{Watson}) to make it robust to high-dimensionality. To do so, consider the (null) \emph{tangent-normal decomposition} 
$\Xb_{ni}=(\Xb'_{ni}\thetab_0)\thetab_0+u_{ni} {\bf  S}_{ni}$,
where
$$
u_{ni}:= \sqrt{1- (\Xb_{ni}\pr \thetab_0)^2}
\quad
\textrm{ and }
\quad
\Sb_{ni}:=\frac{\Xb_{ni}-(\Xb_{ni}'\thetab_0)\thetab_0}{\|\Xb_{ni}-(\Xb_{ni}'\thetab_0)\thetab_0\|}
,
$$
and note that the Watson statistic rewrites 
\begin{eqnarray*}
W_n 
=
 \frac{p_n-1}{\sum_{i=1}^n u_{ni}^2} 
\sum_{i,j=1}^n 
u_{ni}u_{nj} \Sb_{ni}\pr\Sb_{nj}
&=&
 \frac{p_n-1}{\sum_{i=1}^n u_{ni}^2} 
\Bigg( 
\sum_{i=1}^n u_{ni}^2
+
2 \sum_{1\leq i < j\leq n} 
u_{ni}u_{nj} \Sb_{ni}\pr\Sb_{nj}
 \Bigg) 
\nonumber 
\\[2mm] 
&= &
(p_n-1) + 
\frac{2(p_n-1)}{\sum_{i=1}^n u_{ni}^2} 
\sum_{1\leq i < j\leq n} u_{ni} u_{nj} \Sb_{ni}\pr\Sb_{nj} 
.
\end{eqnarray*}
We then introduce the modified statistic
\begin{equation}
\label{Watsonbis}
\tilde{W}_n
:=
\frac{W_n-(p_n-1)}{\sqrt{2(p_n-1)}}
=  
\bigg(
\frac{\sqrt{2(p_n-1)}}{n{\rm E}[u_{n1}^2]}
\sum_{1\leq i < j\leq n} u_{ni} u_{nj} \Sb_{ni}\pr\Sb_{nj} 
\bigg)
\,\Big/\,
\bigg(
\frac{\frac{1}{n}\sum_{i=1}^n u_{ni}^2}{{\rm E}[u_{n1}^2]}
\bigg)
.
\end{equation}
The following result, that provides the $(n,p)$-asymptotic null distribution of~$\tilde{W}_n$, is the main result of the paper.
%
\vspace{2mm}


\begin{theorem}
\label{maintheor}
Let $\Xb_{ni}$, $i=1,\ldots,n$, $n=1,2,\ldots,$ form a triangular array of random vectors satisfying the following conditions~: 
(i) for any $n$, $\Xb_{n1},\Xb_{n2},\ldots,\Xb_{nn}$ are mutually independent and share a common rotationally symmetric distribution on~$\mathcal{S}^{p_n-1}$ with location $\thetab_0$;
(ii) $p_n\to \infty$ as $\ny$;
(iii)  
${\rm E}[u_{n1}^2]>0$ and
(iv) ${\rm E}[u_{n1}^4]/({\rm E}[u_{n1}^2])^2 =o(n)$ as $\ny$.
Then $\tilde{W}_n$ is asymptotically standard normal.
\end{theorem}
\vspace{3mm}

The assumptions of Theorem~\ref{maintheor} are extremely mild. Note in particular that it is not assumed that the common distribution of the $\Xb_{ni}$'s is absolutely continuous with respect to the surface area measure on~$\mathcal{S}^{p_n-1}$. Imposing~(iii) is strictly equivalent to requiring that $\Xb_{n1}\neq \thetab_0$ almost surely, which ensures that the~$\Sb_{ni}$'s are well-defined with probability one. Finally, a sufficient (yet not necessary) condition for~(iv) is that 
$\sqrt{n} \, {\rm E}[u_{n1}^2]
\to\infty$
as~$\ny$. In other words, if~(iv) does not hold, we must then have that, for some constant~$C>0$, 
\begin{equation}
\label{cond}
{\rm E}[(\Xb_{n1}\pr\thetab_0)^2] \geq 1- \frac{C}{\sqrt{n}}
\end{equation} 
for infinitely many~$n$. In the high-dimensional setup considered, (\ref{cond}) is extremely pathological, since it corresponds to the distribution of~$\Xb_{n1}$ concentrating in \emph{one} particular direction --- namely, the direction~$\thetab_0$ --- in the expanding Euclidean space~$\R^{p_n}$. Most importantly, it should be noted that (ii) allows~$p_n$ to go to infinity in an arbitrary way with~$n$, so that Theorem~\ref{maintheor} provides a ``$(n,p)$-universal" asymptotic distribution result for the modified Watson statistic. 

%
%

%
%

The ratio decomposition of~$\tilde{W}_n$ in~(\ref{Watsonbis}) invites to base the proof of Theorem~\ref{maintheor} on the Slutsky Lemma. The stochastic convergence of the denominator is taken care of in 
\vspace{2mm}

\begin{proposition} 
\label{largenum}
Under the assumptions of Theorem~\ref{maintheor}, 
$$
\frac{\frac{1}{n} \sum_{i=1}^{n} u_{ni}^2}{{\rm E}[u_{n1}^2]}
\to 1
$$ 
in quadratic mean as $n\rightarrow\infty$.
\end{proposition}
\vspace{0mm}

\begin{proof}[of Proposition~\ref{largenum}]
Since
\begin{eqnarray*}
\lefteqn{
\hspace{-3mm}
{\rm E}
\Bigg[
\left(\frac{\frac{1}{n} \sum_{i=1}^{n} u_{ni}^2}{{\rm E}[u_{n1}^2]}-1\right)^2
\Bigg]
=
\frac{1}{({\rm E}[u_{n1}^2])^2}
\,
{\rm E}
\Bigg[
\Bigg(
\frac{1}{n} \sum_{i=1}^{n} u_{ni}^2 - {\rm E}[u_{n1}^2]
\Bigg)^2
\Bigg]
}
\\[2mm]
& &
\hspace{5mm}
=
\frac{1}{({\rm E}[u_{n1}^2])^2}
\,
{\rm Var}\Bigg[\frac{1}{n} \sum_{i=1}^{n} u_{ni}^2\Bigg]
=
\frac{{\rm Var}[u_{n1}^2]}{n({\rm E}[u_{n1}^2])^2}
\leq 
 \frac{{\rm E}[u_{n1}^4]}{n({\rm E}[u_{n1}^2])^2},
\end{eqnarray*}
the result follows from Condition~(iv) in Theorem~\ref{maintheor}.
\end{proof}
\vspace{2mm}
 

To establish Theorem~\ref{maintheor}, it is therefore sufficient to prove
\vspace{2mm}

\begin{proposition} 
\label{prop2}
Under the assumptions of Theorem~\ref{maintheor}, 
$$
R_n
:=
\frac{\sqrt{2(p_n-1)}}{n{\rm E}[u_{n1}^2]}
\sum_{1\leq i < j\leq n} u_{ni} u_{nj} \Sb_{ni}\pr\Sb_{nj} 
$$ 
is asymptotically standard normal.
\end{proposition}
\vspace{2mm}

The proof of this proposition is much more delicate and will be based on the following martingale Central Limit Theorem; see Theorem~35.12 in \cite{Bil1995}. 
\vspace{2mm}

\begin{theorem}
\label{Bil}
Assume that, for each $n$, $Z_{n1},Z_{n2},\ldots$ is a martingale relative to the filtration $\mathcal{F}_{n1},\mathcal{F}_{n2},\ldots$ and define $Y_{n\ell}=Z_{n\ell}-Z_{n,\ell-1}$. Suppose that the $Y_{n\ell}$'s have finite second-order moments and let $\sigma^2_{n\ell}={\rm E}[Y_{n\ell}^2\,|\,\mathcal{F}_{n,\ell-1}]$ (with $\mathcal{F}_{n0}=\{\emptyset,\Omega\}$). Assume that $\sum_{\ell=1}^\infty Y_{n\ell}$ and $\sum_{\ell=1}^\infty \sigma^2_{n\ell}$ converge with probability 1. Then, if, for $n\rightarrow\infty$, 
\begin{equation}
\label{35.35}
\sum_{\ell=1}^\infty \sigma^2_{n\ell}=\sigma^2+o_{\rm P}(1),
\end{equation}
where $\sigma$ is a positive real number, and
\begin{equation}
\label{36.36}
\sum_{\ell=1}^\infty {\rm E}\big[Y_{n\ell}^2\,\mathbb{I}[|Y_{n\ell}|\geq\varepsilon]\big]\rightarrow0 \quad\forall\varepsilon>0,
\end{equation}
we have that $\sigma^{-1}\sum_{\ell=1}^\infty Y_{n\ell}$ is asymptotically standard normal.
\end{theorem}
\vspace{2mm}

In order to apply this result, we need to identify the distinct quantities in the present setting. Let ${\cal F}_{n\ell}$ be the $\sigma$-algebra generated by $\Xb_{n1}, \ldots, \Xb_{n \ell}$ and denote by ${\rm E}_{n\ell} [.]$ the conditional expectation with respect to ${\cal F}_{n\ell}$. Then, letting
$$
Y_{n\ell}
:=
{\rm E}_{n\ell} [R_n]-{\rm E}_{n,\ell-1} [R_n] 
=
\frac{\sqrt{2(p_n-1)}}{n {\rm E}[u_{n1}^2]}
\,
 \sum_{i=1}^{\ell-1} u_{ni} u_{n\ell} \Sb_{ni}\pr\Sb_{n\ell}
$$
for $\ell=1,\ldots,n$ and (as in \cite{Bil1995}) $Y_{n\ell}=0$ for~$\ell>n$, 
 we clearly have that
$R_n= \sum_{\ell=2}^nY_{n\ell}$, where the $Y_{n\ell}$'s have finite second-order moments. Also, $\sum_{\ell=2}^\infty Y_{n\ell}=\sum_{\ell=2}^n Y_{n\ell}$ and $\sum_{\ell=2}^\infty \sigma^2_{n\ell}=\sum_{\ell=2}^n \sigma^2_{n\ell}$, with $\sigma^2_{n\ell}={\rm E}_{n,\ell-1} [Y_{n\ell}^2]$ as in Theorem \ref{Bil}, and both converge with probability~1, as required. Now, the crucial conditions~\eqref{35.35} and~\eqref{36.36} are shown to hold in the subsequent lemmas (see the Appendix for the proofs).

\begin{lemma}
\label{firstlem}
Under the assumptions of Theorem~\ref{maintheor}, 
$\sum_{\ell=2}^n \sigma^2_{n\ell}\to 1$
in quadratic mean as~$n\to~\infty$. 
\end{lemma} 

\begin{lemma}
\label{THElemma}
Under the assumptions of Theorem~\ref{maintheor}, 
$
\sum_{\ell=2}^n {\rm E}[Y_{n\ell}^2 \; {\mathbb I}[| Y_{n\ell}| > \varepsilon]]\to 0
$ 
as~$n\to \infty$ for any~$\varepsilon>0$.   
\end{lemma}

These lemmas allow to use Theorem~\ref{Bil} to prove Proposition~\ref{prop2} which, jointly with Proposition~\ref{largenum}, establishes Theorem~\ref{maintheor}. Clearly, the resulting high-dimensional Watson test, $\tilde{\phi}^W_n$, say, rejects the null hypothesis $\mathcal{H}_0:\thetab=\thetab_0$ in favor of $\mathcal{H}_1:\thetab\neq\thetab_0$ at asymptotic level $\alpha$ whenever 
$$
\tilde{W}_n > \Phi^{-1}(1-\alpha),
$$
where $\Phi$ denotes the cumulative distribution function of the standard normal distribution. As already pointed out when commenting the assumptions of Theorem~\ref{maintheor}, this test achieves asymptotic null size~$\alpha$ irrespective of the way~$p_n$ goes to infinity with~$n$.

For the problem considered above, \cite{PaiVer2013b} introduced the high-dimensional \emph{sign} statistic 
\begin{equation}
\label{signteststat}
\tilde{S}_n
:=
\frac{\sqrt{2(p_n-1)}}{n}
\sum_{1\leq i < j\leq n} \Sb_{ni}\pr\Sb_{nj} 
\end{equation}
and showed that the $(n,p)$-universal asymptotic null distribution of~$\tilde{S}_n$ is standard normal. In the next result, we identify assumptions on the sequence~$u_{n1}$ under which $\tilde{W}_n$ and~$\tilde{S}_n$ are ($(n,p)$-universally) asymptotically equivalent in probability under the null. 
\vspace{2mm}

\begin{theorem}
\label{compare}
Let the assumptions of Theorem~\ref{maintheor} hold and further assume that (v) ${\rm E}[u_{n1}^2]/({\rm E}[u_{n1}])^2 \to 1$ as~$\ny$. Then,
$
\tilde{W}_n-\tilde{S}_n=o_{\rm P}(1)
$
as~$\ny$.
\end{theorem}

\begin{proof}[of Theorem~\ref{compare}]
Decompose $\tilde{W}_n-\tilde{S}_n$ into $A_n+B_n$, with
$$
A_n
=
\bigg(
\frac{{\rm E}[u_{n1}^2]}{\frac{1}{n}\sum_{i=1}^n u_{ni}^2}-1
\bigg)
\,
\frac{\sqrt{2(p_n-1)}}{n{\rm E}[u_{n1}^2]}
\sum_{1\leq i < j\leq n} u_{ni} u_{nj} \Sb_{ni}\pr\Sb_{nj} 
$$
and
$$
B_n
=
\frac{\sqrt{2(p_n-1)}}{n}
\sum_{1\leq i < j\leq n} \bigg(\frac{u_{ni} u_{nj}}{{\rm E}[u_{n1}^2]}-1\bigg) \Sb_{ni}\pr\Sb_{nj} 
.
$$
Propositions~\ref{largenum} and~\ref{prop2} readily entail that~$A_n=o_{\rm P}(1)
$
as~$\ny$. As for~$B_n$, we have (see the beginning of the Appendix for a recall on some results regarding expectations of the signs $\Sb_{ni}$)
\begin{eqnarray*}
\lefteqn{
{\rm E}[B_n^2]
=
\frac{2(p_n-1)}{n^2}
\!\!
\sum_{1\leq i < j\leq n} 
\!
{\rm E}
\Bigg[
\bigg(\frac{u_{ni} u_{nj}}{{\rm E}[u_{n1}^2]}-1\bigg)^2 \! (\Sb_{ni}\pr\Sb_{nj})^2 
\Bigg]
=
\frac{2}{n^2}
\!
\sum_{1\leq i < j\leq n} 
\!
{\rm E}
\Bigg[
\bigg(\frac{u_{ni} u_{nj}}{{\rm E}[u_{n1}^2]}-1\bigg)^2  
\Bigg]
}
\\[3mm]
& &
=
\frac{n-1}{n}
\,{\rm E}
\Bigg[
\bigg(\frac{u_{n1} u_{n2}}{{\rm E}[u_{n1}^2]}-1\bigg)^2  
\Bigg]
=
\frac{2(n-1)}{n}
\,{\rm E}
\Bigg[
 1-  \frac{u_{n1} u_{n2}}{{\rm E}[u_{n1}^2]}
\Bigg]
=
\frac{2(n-1)}{n}
\Bigg(
 1-  \frac{({\rm E}[u_{n1}])^2}{{\rm E}[u_{n1}^2]}
\Bigg),
\end{eqnarray*}
which, in view of Condition~(v), is $o(1)$ as~$\ny$. The result follows. 
\end{proof}

This result shows that, quite intuitively, if $u_{n1}$ becomes constant asymptotically (in the sense that 
${\rm Var}[u_{n1}]/({\rm E}[u_{n1}])^2 \to 0$), then the high-dimensional Watson test~$\tilde{\phi}^W_n$ coincides with the sign test based on~(\ref{signteststat}). This should be considered as the exception rather than the rule, though, since there is no particular reason why the distribution of~$\Xb_{n1}$ should concentrate in (a possibly translated version of) the orthogonal complement of~$\thetab_0$.

\section{Spiked covariance matrices} \label{spike}

Let $\Yb_{n1}, \ldots, \Yb_{nn}$ be a random sample from the $p_n$-dimensional multinormal distribution with mean zero and covariance matrix~$\Sigb$. For fixed~$\thetab_0\in{\cal S}^{p_n-1}$, we consider here the problem of testing the null hypothesis that~$\Sigb$ has a  ``$\thetab_0$-spiked" structure, that is, is of the form 
$$
{\cal H}_0^{\rm spi}: \Sigb= \sigma^2({\bf I}_{p_n}+ \lambda \thetab_0 \thetab_0\pr),
\
\textrm{ for some } \sigma^2,\lambda>0.
$$

Consider  the projections~$\Xb_{ni}:=\Yb_{ni}/\|\Yb_{ni}\|$, $i=1,\ldots,n$, of the observations on the unit hypersphere, and let
$$
\Sb_{ni}:=\frac{\Xb_{ni}-(\Xb_{ni}'\thetab_0)\thetab_0}{\|\Xb_{ni}-(\Xb_{ni}'\thetab_0)\thetab_0\|}.
$$
Under~${\cal H}_0^{\rm spi}$, (i) the $\Sb_{ni}$'s are mutually independent and are uniformly distributed over~${\mathcal S}^{p_n-1}(\thetab_0^\perp):=\{ \xb \in\mathcal{S}^{p_n-1}\,|\, \xb\pr\thetab_0=0 \}$; moreover, (ii) the $\Xb_{ni}'\thetab_0$'s are independent and identically distributed, and they are independent of the $\Sb_{ni}$'s. It is well-known that (i)-(ii) imply that the common distribution of the projected observations~$\Xb_{ni}$ is rotationally symmetric about~$\thetab_0$. Consequently, a high-dimensional test for $\thetab_0$-spikedness is the test, $\tilde{\phi}^{\rm spi}_n$ say, that rejects the null~${\cal H}_0^{\rm spi}$, at asymptotic level~$\alpha$, whenever 
$$
\tilde{W}^{\rm spi}_n(\Yb_{n1},\ldots,\Yb_{nn})
:=
\tilde{W}_n(\Xb_{n1},\ldots,\Xb_{nn})
>
\Phi^{-1}(1-\alpha)
.
$$ 

Theorem~\ref{maintheor} ensures that~$\tilde{\phi}^{\rm spi}_n$ has asymptotic null size~$\alpha$ as soon as~$p_n$ goes to infinity with~$n$ (universal $(n,p)$ asymptotics), which is illustrated in the simulations of the next section. Typically, this test will show large powers against $\thetab$-spiked alternatives, with $\thetab \neq \thetab_0$. 
\vspace{-5mm}

\section{Monte Carlo study}\label{simus}

In this section, we conduct a Monte Carlo simulation study to check the validity of our universal asymptotic results related to both~$\tilde{W}_n$ and~$\tilde{W}^{\rm spi}_n$. To do so, we generated, for every~$(n,p)\in C\times C$, with $C=\{5, 30, 200, 1,\!000\}$, and for  $\thetab_0$  the first vector of the canonical basis of~$\R^p$, $M=2,\!500$ independent random samples from each of the following $p$-dimensional distributions~:
\begin{itemize}
\item[(i)] 
the FvML distribution~$\mathcal{R}(\thetab_0,F_{p,2})$ (see Section~\ref{prelim});
\item[(ii)]  the Purkayastha distribution~$\mathcal{R}(\thetab_0,G_{p,1})$, associated with~$G_{p,\kappa}(t)=d_{p,\kappa}\int_{-1}^t (1-s^2)^{(p-3)/2} \exp(-\kappa\arccos(s)) \,ds$ ($t\in[-1,1]$), where~$d_{p,\kappa}$ is a normalizing constant; 
\item[(iii)] 
 the multinormal distribution with mean zero and covariance matrix $\Sigb={\bf I}_p+ (1/2) \thetab_0 \thetab_0\pr$.
\end{itemize}
The modified Watson statistic~$\tilde{W}_n$ was evaluated on the samples from~(i)-(ii) (rotational symmetry about~$\thetab_0$), while the statistic~$\tilde{W}_n^{\rm spi}$ was computed for each sample from~(iii) ($\thetab_0$-spikedness). For each~$(n,p)$ regime considered, we report the corresponding histograms of~$\tilde{W}_n$ and~$\tilde{W}_n^{\rm spi}$ in Figures~\ref{fig1}-\ref{fig2} and in Figure~\ref{fig3}, respectively (each histogram is based on $M=2,500$ values of these statistics).

From Theorem~\ref{maintheor} and the discussion in Section~\ref{spike}, histograms are expected to be approximately standard normal as soon as~$\min(n,p)$ is large, in a universal way (that is, irrespective of the relative size of~$n$ and~$p$). Inspection of the results shows that, for all three setups, the  standard normal approximation is valid for moderate to large values of~$n$ and~$p$, irrespective of the value of~$p/n$, which confirms our universal asymptotic results. Note also that, for small $p$ and moderate to large~$n$ (that is, $p=5$ and $n\geq 30$), histograms are approximately (standardized) chi-square, which is consistent with classical fixed-$p$ asymptotic results; see Section~\ref{prelim}.

\section*{Acknowledgement}

Christophe Ley thanks the Fonds National de la Recherche Scientifique, Communaut\'e Fran\c caise de Belgique, for support via a Mandat de Charg\'e de Recherche. Davy Paindaveine's research is supported by an \mbox{A.R.C.} contract from the Communaut\'e Fran\c{c}aise de Belgique and by the IAP research network grant \mbox{nr.} P7/06 of the Belgian government (Belgian Science Policy).

\appendix

\section*{Appendix: proofs of Lemmas~\ref{firstlem} and~\ref{THElemma}}
\vspace{3mm}

We recall that, under the assumptions of Theorem~\ref{maintheor}, the signs $\Sb_{ni}$ are uniformly distributed over~${\mathcal S}^{p_n-1}(\thetab_0^\perp)$ (see Section~\ref{spike}) and that the $u_{ni}$'s are independent of the $\Sb_{ni}$'s, $i=1,\ldots,n$. From Lemma~A.1 in \cite{PaiVer2013b} it directly follows that, for fixed~$n$, the quantities $\rho_{n,ij}:=\Sb_{ni}'\Sb_{nj}$ are pairwise independent and satisfy~${\rm E}[\rho_{n,ij}]=0$, ${\rm E}[\rho_{n,ij}^2]=1/(p_n-1)$, and ${\rm E}[\rho_{n,ij}^4]=3/(p_n^2-1)$. 
 \vspace{2mm}

%

\begin{proof}[of Lemma~\ref{firstlem}]
Rotational symmetry about~$\thetab_0$ readily yields ${\rm E}[\Sb_{n\ell}\Sb\pr_{n\ell}]=\frac{1}{p_n-1} ({\bf I}_{p_n} - \thetab_0 \thetab_0\pr)$. The independence between the $u_{ni}$'s and $\Sb_{ni}$'s then provides
$$
\sigma^2_{n\ell}={\rm E}_{n,\ell-1}[Y_{n\ell}^2] 
=
\frac{2(p_n-1)}{n^2 ({\rm E}[u_{n1}^2])^2}  \sum_{i,j=1}^{\ell-1} u_{ni} u_{nj} {\rm E}[u_{n\ell}^2] \Sb_{ni}\pr {\rm E}[\Sb_{n\ell}\Sb_{n\ell}\pr]
\Sb_{nj} 
=
\frac{2}{n^2 {\rm E}[u_{n1}^2]}  \sum_{i,j=1}^{\ell-1} u_{ni} u_{nj} \rho_{n,ij}. 
$$
Hence we obtain
\begin{equation}
\label{exps}
{\rm E}\Bigg[\sum_{\ell=2}^n \sigma^2_{n\ell}\Bigg] 
=
\frac{2}{n^2 {\rm E}[u_{n1}^2]} \sum_{\ell=2}^n  \sum_{i,j=1}^{\ell-1} {\rm E} [u_{ni}u_{nj}]  {\rm E} [\rho_{n,ij}] 
=
\frac{2}{n^2} \sum_{\ell=2}^n\,  (\ell-1)
=\frac{n-1}{n}
.
\end{equation}
Moreover, the pairwise independence of the~$\rho_{n,ij}$'s entails 
\begin{equation*}
{\rm Var}\Bigg[\sum_{\ell=2}^n \sigma^2_{n\ell}\Bigg] 
=
\frac{4}{n^4 ({\rm E}[u_{n1}^2])^2} 
{\rm Var}\Bigg[\sum_{\ell=2}^n  \sum_{i,j=1}^{\ell-1} u_{ni} u_{nj}  \rho_{n,ij} \Bigg]
=
\frac{4}{n^4 ({\rm E}[u_{n1}^2])^2}\left\{ T_1\n +  4\, T_2\n\right\}
,
\label{here1}
\end{equation*}
with
$$
T_1\n
:=
{\rm Var}\Bigg[\sum_{\ell=2}^n  \sum_{i=1}^{\ell-1} u_{ni}^2 \Bigg]
=
{\rm Var}\Bigg[\sum_{i=1}^{n-1} \, (n-i) u_{ni}^2 \Bigg] 
=
\sum_{i=1}^{n-1} \, (n-i)^2 \,
 {\rm Var}[u_{n1}^2]
\leq 
n^3
\,
 {\rm Var}[u_{n1}^2]
$$
and
\begin{eqnarray*}
T_2\n
&:=&
{\rm Var}\Bigg[\sum_{\ell=2}^n  \sum_{1\leq i<j\leq \ell-1} u_{ni} u_{nj} \rho_{n,ij}\Bigg]
=
 {\rm Var}\Bigg[  \sum_{1\leq i<j\leq n-1}(n-j) u_{ni} u_{nj} \rho_{n,ij} \Bigg]
\\[3mm]
& = &
\sum_{1\leq i<j\leq n-1} (n-j)^2 {\rm Var}[u_{ni} u_{nj}\rho_{n,ij}]
=
\sum_{1\leq i<j\leq n-1} (n-j)^2 {\rm E}[u_{ni}^2 u^2_{nj}\rho^2_{n,ij}]
\\[3mm]
& = &
\frac{({\rm E}[u_{n1}^2])^2}{p_n-1}
\sum_{1\leq i<j\leq n-1} (n-j)^2 
 \leq 
\frac{n^4 ({\rm E}[u_{n1}^2])^2}{p_n-1}
.
\end{eqnarray*}
Hence,
\begin{equation}
\label{vars}
{\rm Var}\Bigg[\sum_{\ell=2}^n \sigma^2_{n\ell}\Bigg] 
\leq
\frac{ 4{\rm Var}[u_{n1}^2]}{ n({\rm E}[u_{n1}^2])^2} +\frac{16}{p_n-1}
\leq 
\frac{ 4{\rm E}[u_{n1}^4]}{ n({\rm E}[u_{n1}^2])^2} +\frac{16}{p_n-1}
\to 0,
\end{equation}
in view of Conditions~(ii) and (iv) from Theorem~\ref{maintheor}. 
Using~(\ref{exps}) and~(\ref{vars}) in 
$$
{\rm E}\Bigg[\Bigg(\sum_{\ell=2}^n \sigma^2_{n\ell}-1\Bigg)^2\Bigg]
=
{\rm Var}\Bigg[\sum_{\ell=2}^n \sigma^2_{n\ell}\Bigg] 
+
\Bigg( {\rm E}\Bigg[\sum_{\ell=2}^n \sigma^2_{n\ell}-1\Bigg] \Bigg)^2
$$
then establishes the result.
\end{proof}
\vspace{2mm}


\begin{proof}[of Lemma~\ref{THElemma}]
Applying first the Cauchy-Schwarz inequality, then the Chebyshev inequality, yields
$$
\sum_{\ell=2}^n {\rm E}[Y_{n\ell}^2 \; {\mathbb I}[| Y_{n\ell}| > \varepsilon]]
\leq 
 \sum_{\ell=2}^n 
\sqrt{{\rm E}[Y_{n\ell}^4]} 
\,
\sqrt{{\rm P}[|Y_{n\ell}| > \varepsilon]} 
\leq 
\frac{1}{\varepsilon}
 \sum_{\ell=2}^n 
\sqrt{{\rm E}[Y_{n\ell}^4]} 
\,
\sqrt{{\rm Var}[Y_{n\ell}]} 
.
$$
Noting that
$
{\rm Var}[Y_{n\ell}]
\leq 
{\rm E}[Y^2_{n\ell}]
=2(\ell-1)/n^2
$, we obtain
\begin{equation}
\sum_{\ell=2}^n {\rm E}[Y_{n\ell}^2 \; {\mathbb I}[| Y_{n\ell}| > \varepsilon]]
\leq 
\frac{\sqrt{2}}{\varepsilon n}
 \sum_{\ell=2}^n 
\sqrt{\ell\,
{\rm E}[Y_{n\ell}^4]} 
.
\label{main}
\end{equation}
Using the fact that $0\leq u_{ni} \leq1$ almost surely and the independence between the $u_{ni}$'s and the $\Sb_{ni}$'s, we get
\begin{eqnarray*}
\lefteqn{
\hspace{-9mm}
{\rm E}\Bigg[ 
\bigg( \sum_{i=1}^{\ell-1} u_{ni} u_{n\ell} \rho_{n,i\ell} \bigg)^4 
\Bigg] 
=\
\sum_{i,j,r,s=1}^{\ell-1}
{\rm E} \big[u_{n\ell}^4 u_{ni} u_{nj} u_{nr} u_{ns} \rho_{n,i\ell} \rho_{n,j\ell}\rho_{n,r\ell}\rho_{n,s\ell}\big] 
}
\\[4mm]
& &
\hspace{3mm}
=\
(\ell-1)({\rm E}[u_{n1}^4])^2
{\rm E} \big[ \rho_{n,1\ell}^4\big] 
+
3(\ell-1)(\ell-2){\rm E}[u_{n1}^4] ({\rm E}[u_{n1}^2])^2
{\rm E} \big[ \rho_{n,1\ell}^2 \rho_{n,2\ell}^2 \big] 
\nonumber
\\[4mm]
& &
\hspace{3mm}
=\
\frac{3(\ell-1)}{p^2_n-1}({\rm E}[u_{n1}^4])^2
+
\frac{3(\ell-1)(\ell-2)}{(p_n-1)^2}{\rm E}[u_{n1}^4] ({\rm E}[u_{n1}^2])^2
\\[2mm]
& &
\hspace{3mm}
\leq 
\frac{3}{(p_n-1)^2} 
\Big[
\ell ({\rm E}[u_{n1}^4])^2
+
\ell^2 
{\rm E}[u_{n1}^4] ({\rm E}[u_{n1}^2])^2
\Big]
\label{ok}
\end{eqnarray*}
which yields
\begin{eqnarray*}
{\rm E}\big[ 
Y_{n\ell}^4 
\big]
&\leq&
\frac{4(p_n-1)^2}{n^4({\rm E}[u_{n1}^2])^4}
\times
\frac{3}{(p_n-1)^2} 
\Big[
\ell ({\rm E}[u_{n1}^4])^2
+
\ell^2 
{\rm E}[u_{n1}^4] ({\rm E}[u_{n1}^2])^2
\Big]
\\[2mm]
&\leq&
\frac{12}{n^4}
\Bigg[
\ell 
\,
\frac{({\rm E}[u_{n1}^4])^2}{({\rm E}[u_{n1}^2])^4}
+
\ell^2 
\,
\frac{{\rm E}[u_{n1}^4]}{({\rm E}[u_{n1}^2])^2}
\Bigg]
.
\end{eqnarray*}
Plugging into~(\ref{main}), we conclude that
\begin{eqnarray*}
\sum_{\ell=2}^n {\rm E}[Y_{n\ell}^2 \; {\mathbb I}[| Y_{n\ell}| > \varepsilon]]
&\leq& 
\frac{\sqrt{24}}{\varepsilon n^3}
 \sum_{\ell=2}^n 
\sqrt{
\ell^2 
\,
\frac{({\rm E}[u_{n1}^4])^2}{({\rm E}[u_{n1}^2])^4}
+
\ell^3 
\,
\frac{{\rm E}[u_{n1}^4]}{({\rm E}[u_{n1}^2])^2}
} 
\\[2mm]
&\leq & 
\frac{\sqrt{24}}{\varepsilon n^3}
 \sum_{\ell=2}^n 
\bigg( 
\ell
\,
\frac{{\rm E}[u_{n1}^4]}{({\rm E}[u_{n1}^2])^2}
+
\ell^{3/2} 
\,
\sqrt{\frac{{\rm E}[u_{n1}^4]}{({\rm E}[u_{n1}^2])^2}}
\bigg) 
\\[2mm]
&\leq& 
O(n^{-1})
\,
\frac{{\rm E}[u_{n1}^4]}{({\rm E}[u_{n1}^2])^2}
+
O(n^{-1/2})
\,
\sqrt{\frac{{\rm E}[u_{n1}^4]}{({\rm E}[u_{n1}^2])^2}},
\end{eqnarray*}
which, in view of Condition~(iv) from Theorem~\ref{maintheor}, is indeed $o(1)$. 
\end{proof}
\vspace{4mm}


{
\bibliographystyle{biometrika}
\bibliography{HDMean_neutral}

\providecommand{\noopsort}[1]{}
\begin{thebibliography}{27}
\expandafter\ifx\csname natexlab\endcsname\relax\def\natexlab#1{#1}\fi

\bibitem[{Banerjee et~al.(2003)Banerjee, Dhillon, Ghosh \&
  Sra}]{banerjee2003generative}
\textsc{Banerjee, A.}, \textsc{Dhillon, I.}, \textsc{Ghosh, J.} \& \textsc{Sra,
  S.} (2003).
\newblock Generative model-based clustering of directional data.
\newblock \textit{{\rm In} Proceedings of the ninth ACM SIGKDD international
  conference on Knowledge discovery and data mining} \textbf{\hspace{-1mm}},
  19--28.

\bibitem[{Banerjee et~al.(2005)Banerjee, Dhillon, Ghosh \& Sra}]{Banetal2005}
\textsc{Banerjee, A.}, \textsc{Dhillon, I.}, \textsc{Ghosh, J.} \& \textsc{Sra,
  S.} (2005).
\newblock Clustering on the unit hypersphere using von mises-fisher
  distributions.
\newblock \textit{J. Mach. Learn. Res.} \textbf{6}, 1345--1382.

\bibitem[{Banerjee \& Ghosh(2002)}]{BanGho2002}
\textsc{Banerjee, A.} \& \textsc{Ghosh, J.} (2002).
\newblock Frequency sensitive competitive learning for clustering on
  high-dimensional hyperspheres.
\newblock \textit{{\rm In} Proceedings International Joint Conference on Neural
  Networks} \textbf{\hspace{-1mm}}, 1590--1595.

\bibitem[{Banerjee \& Ghosh(2004)}]{BanGho2004}
\textsc{Banerjee, A.} \& \textsc{Ghosh, J.} (2004).
\newblock Frequency sensitive competitive learning for scalable balanced
  clustering on high-dimensional hyperspheres.
\newblock \textit{IEEE T. Neural Networ.} \textbf{15}, 702--719.

\bibitem[{Billingsley(1995)}]{Bil1995}
\textsc{Billingsley, P.} (1995).
\newblock \textit{Probability and Measure}.
\newblock New York, Chichester: Wiley, 3rd ed.

\bibitem[{Cai et~al.(2013)Cai, Fan \& Jiang}]{Caietal2013}
\textsc{Cai, T.}, \textsc{Fan, J.} \& \textsc{Jiang, T.} (2013).
\newblock Distributions of angles in random packing on spheres.
\newblock \textit{J. Mach. Learn. Res.} \textbf{14}, 1837--1864.

\bibitem[{Cai \& Jiang(2012)}]{CaJia12}
\textsc{Cai, T.} \& \textsc{Jiang, T.} (2012).
\newblock Phase transition in limiting distributions of coherence of
  high-dimensional random matrices.
\newblock \textit{J. Multivariate Anal.} \textbf{107}, 24--39.

\bibitem[{Chen \& Qin(2010)}]{CheQin2010}
\textsc{Chen, S.} \& \textsc{Qin, Y.} (2010).
\newblock A two-sample test for high-dimensional data with applications to
  gene-set testing.
\newblock \textit{Ann. Statist.} \textbf{38}, 808--835.

\bibitem[{Chen et~al.(2010)Chen, Zhang \& Zhong}]{Cheetal2010}
\textsc{Chen, S.~X.}, \textsc{Zhang, L.-X.} \& \textsc{Zhong, P.-S.} (2010).
\newblock Tests for high-dimensional covariance matrices.
\newblock \textit{J. Amer. Statist. Assoc.} \textbf{105}, 810--819.

\bibitem[{Dryden(2005)}]{Dry2005}
\textsc{Dryden, I.~L.} (2005).
\newblock Statistical analysis on high-dimensional spheres and shape spaces.
\newblock \textit{Ann. Statist.} \textbf{33}, 1643--1665.

\bibitem[{Jiang \& Yang(2013)}]{JiaYan2013}
\textsc{Jiang, T.} \& \textsc{Yang, F.} (2013).
\newblock Central limit theorems for classical likelihood ratio tests for
  high-dimensional normal distributions.
\newblock \textit{Ann. Statist.} \textbf{41}, 2029--2074.

\bibitem[{Johnstone(2001)}]{johnstone2001}
\textsc{Johnstone, I.~M.} (2001).
\newblock On the distribution of the largest eigenvalue in principal components
  analysis.
\newblock \textit{Ann. Statist.} \textbf{29}, 295--327.

\bibitem[{Ledoit \& Wolf(2002)}]{LedWol2002}
\textsc{Ledoit, O.} \& \textsc{Wolf, M.} (2002).
\newblock Some hypothesis tests for the covariance matrix when the dimension is
  large compared to the sample size.
\newblock \textit{Ann. Statist.} \textbf{30}, 1081--1102.

\bibitem[{Li \& Chen(2012)}]{LiChe2012}
\textsc{Li, J.} \& \textsc{Chen, S.~X.} (2012).
\newblock Two sample tests for high-dimensional covariance matrices.
\newblock \textit{Ann. Statist.} \textbf{40}, 908--940.

\bibitem[{Mardia \& Jupp(2000)}]{MarJup2000}
\textsc{Mardia, K.~V.} \& \textsc{Jupp, P.~E.} (2000).
\newblock \textit{Directional Statistics}.
\newblock John Wiley \& Sons.

\bibitem[{Onatski et~al.(2013)Onatski, Moreira \& Hallin}]{Onaetal2013}
\textsc{Onatski, A.}, \textsc{Moreira, M.} \& \textsc{Hallin, M.} (2013).
\newblock Asymptotic power of sphericity tests for high-dimensional data.
\newblock \textit{Ann. Statist.} \textbf{41}, 1204--1231.

\bibitem[{Paindaveine \& Verdebout(2013{\natexlab{a}})}]{PaiVer2013}
\textsc{Paindaveine, D.} \& \textsc{Verdebout, T.} (2013{\natexlab{a}}).
\newblock Optimal rank-based tests for the location parameter of a rotationally
  symmetric distribution on the hypersphere.
\newblock \textit{ECARES Working Paper 2013-36} .

\bibitem[{Paindaveine \& Verdebout(2013{\natexlab{b}})}]{PaiVer2013b}
\textsc{Paindaveine, D.} \& \textsc{Verdebout, T.} (2013{\natexlab{b}}).
\newblock Universal asymptotics for high-dimensional sign tests.
\newblock \textit{ECARES Working Paper 2013-40} .

\bibitem[{Saw(1978)}]{Saw1978}
\textsc{Saw, J.~G.} (1978).
\newblock A family of distributions on the $m$-sphere and some hypothesis
  tests.
\newblock \textit{Biometrika} \textbf{65}, 69--73.

\bibitem[{Schott(2007)}]{Sch2007}
\textsc{Schott, J.} (2007).
\newblock Some high-dimensional tests for a one-way manova.
\newblock \textit{J. Multivariate Anal.} \textbf{98}, 1825--1839.

\bibitem[{Srivastava \& Fujikoshi(2006)}]{SriFuj2006}
\textsc{Srivastava, M.~S.} \& \textsc{Fujikoshi, Y.} (2006).
\newblock Multivariate analysis of variance with fewer observations than the
  dimension.
\newblock \textit{J. Multivariate Anal.} \textbf{97}, 1927--1940.

\bibitem[{Srivastava et~al.(2013)Srivastava, Katayama \& Kano}]{Srietal2013}
\textsc{Srivastava, M.~S.}, \textsc{Katayama, S.} \& \textsc{Kano, Y.} (2013).
\newblock A two sample test in high dimensional data.
\newblock \textit{J. Multivariate Anal.} \textbf{114}, 349--358.

\bibitem[{Srivastava \& Kubokawa(2013)}]{SriKub2013}
\textsc{Srivastava, M.~S.} \& \textsc{Kubokawa, T.} (2013).
\newblock Tests for multivariate analysis of variance in high dimension under
  non-normality.
\newblock \textit{J. Multivariate Anal.} \textbf{115}, 204--216.

\bibitem[{Stam(1982)}]{Sta1982}
\textsc{Stam, A.~J.} (1982).
\newblock Limit theorems for uniform distributions on spheres in
  high-dimensional euclidean spaces.
\newblock \textit{J. Appl. Probab.} \textbf{19}, 221--228.

\bibitem[{Watson(1983{\natexlab{a}})}]{Wat1983a}
\textsc{Watson, G.~S.} (1983{\natexlab{a}}).
\newblock Limit theorems on high-dimensional spheres and stiefel manifolds.
\newblock In \textit{Studies in Econometrics, Time Series, and Multivariate
  Statistics}, S.~Karlin, T.~Amemiya \& L.~A. Goodman, eds. New York: Academic
  Press, 559--570.

\bibitem[{Watson(1983{\natexlab{b}})}]{Wat1983b}
\textsc{Watson, G.~S.} (1983{\natexlab{b}}).
\newblock \textit{Statistics on Spheres}.
\newblock New York: Wiley.

\bibitem[{Watson(1988)}]{Wat1988}
\textsc{Watson, G.~S.} (1988).
\newblock The langevin distribution on high dimensional spheres.
\newblock \textit{J. Appl. Statist.} \textbf{15}, 123--130.

\end{thebibliography}
}

\begin{figure}[htbp!]
\begin{center}
\includegraphics[height=15cm, width=\linewidth]{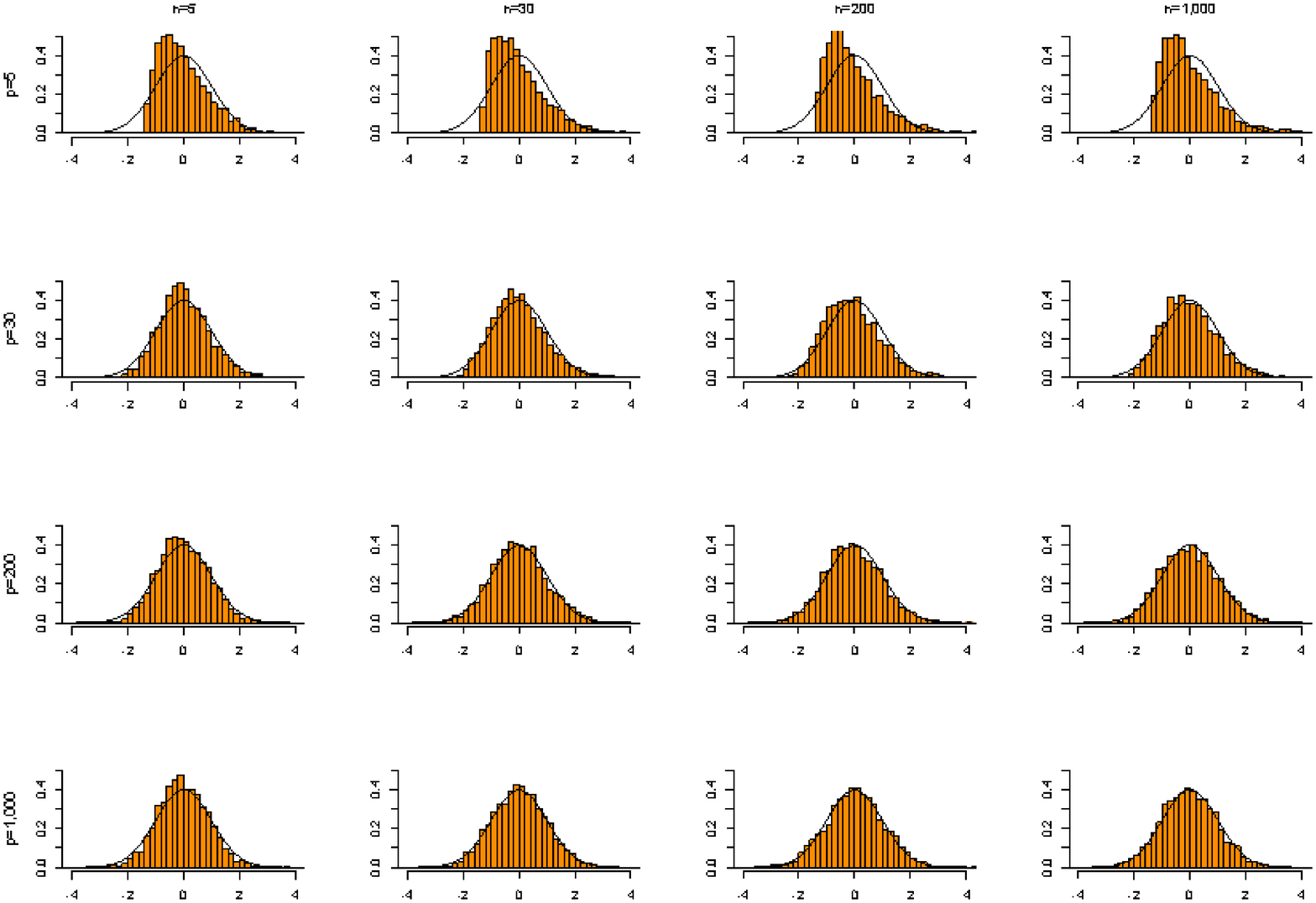}
\caption{Histograms, for various values of~$n$ and~$p$, of the modified Watson statistic~$\tilde{W}_n$ evaluated on $M=2,500$ random samples of size~$n$ from the $p$-dimensional FvML distribution with concentration~$\kappa=2$; see Section~\ref{simus} for details.}
\label{fig1}
\end{center}
\end{figure}

\begin{figure}[htbp!]
\begin{center}
\includegraphics[height=15cm, width=\linewidth]{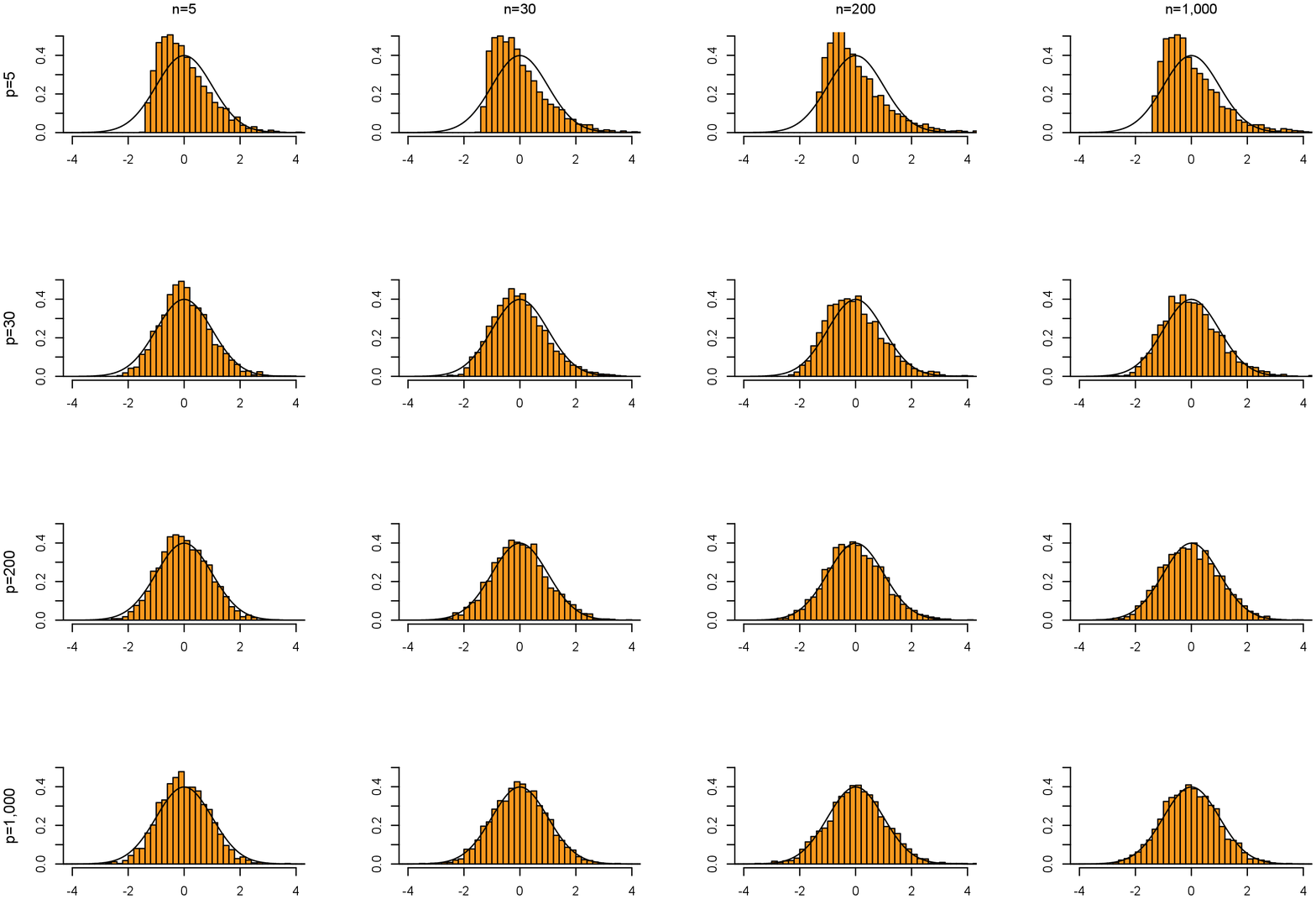}
\caption{Histograms, for various values of~$n$ and~$p$, of the modified Watson statistic~$\tilde{W}_n$ evaluated on $M=2,500$ random samples of size~$n$ from the $p$-dimensional Purkayastha distribution with concentration~$\kappa=1$; see Section~\ref{simus} for details.}
\label{fig2}
\end{center}
\end{figure}

\begin{figure}[htbp!]
\begin{center}
\includegraphics[height=15cm, width=\linewidth]{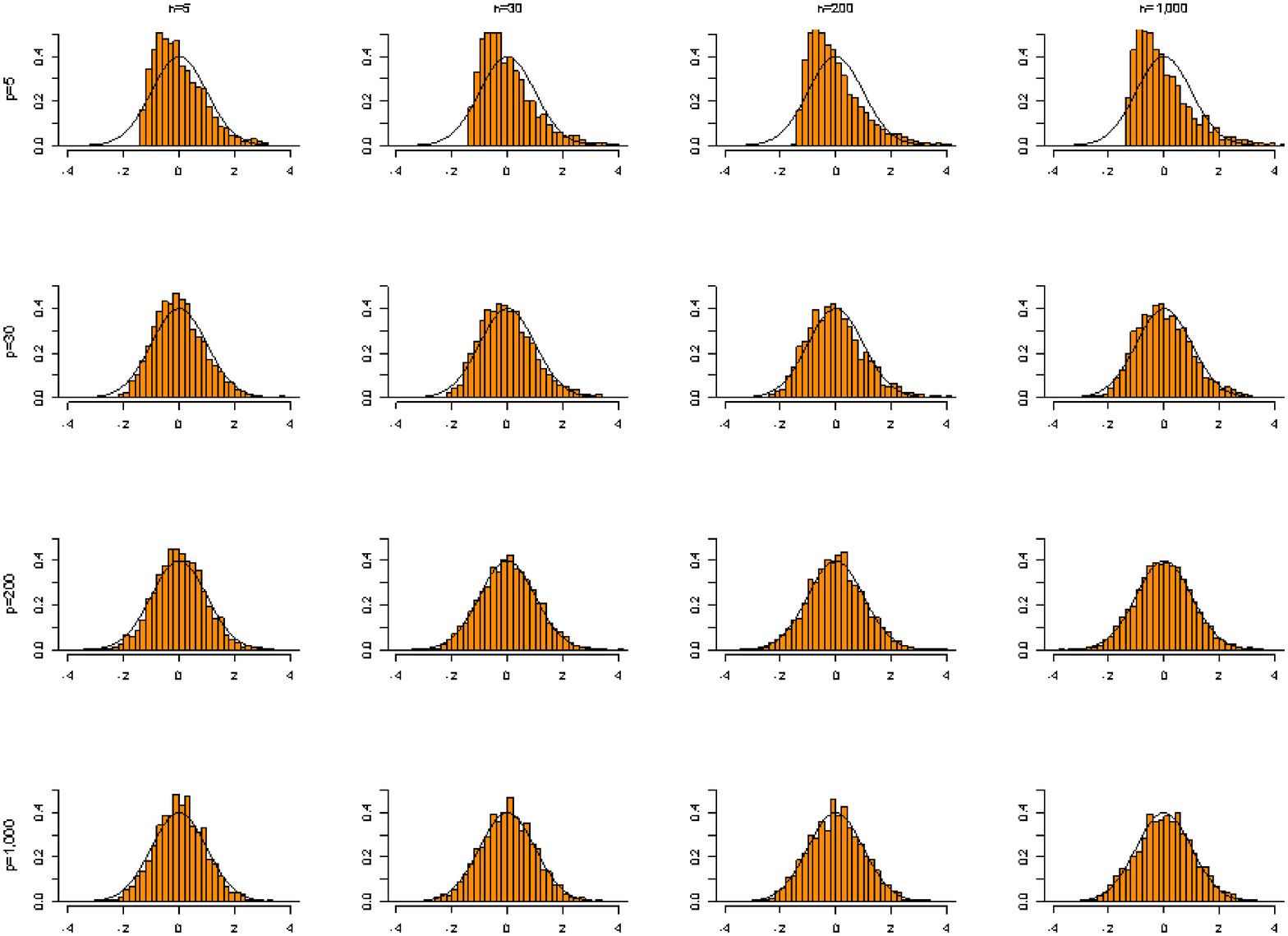}
\caption{\protect Histograms, for various values of~$n$ and~$p$, of the test statistic~$\tilde{W}^{\rm spi}_n$ for $\thetab_0$-spikedness evaluated on $M=2,500$ random samples of size~$n$ from the $p$-dimensional multinormal distribution with mean zero and covariance matrix $\Sigb={\bf I}_p+ (1/2) \thetab_0 \thetab_0\pr$; see Section~\ref{simus} for details.}
\label{fig3}
\end{center}
\end{figure}


\end{document}